\newif\ifpictures
\picturestrue

\documentclass[12pt]{amsart}
\usepackage{amssymb}
\usepackage{amsmath}
\usepackage{amscd}
\usepackage[dvips]{graphicx}
\usepackage{caption}
\usepackage{subcaption}
\usepackage{book tabs}
\usepackage{hyperref}
\usepackage{dsfont}
\usepackage{float}

\ifpictures
\usepackage{psfrag}
\fi

\numberwithin{equation}{section}
\newtheorem{thm}{Theorem}
\newtheorem{prop}[thm]{Proposition}
\newtheorem{lemma}[thm]{Lemma}
\newtheorem{cor}[thm]{Corollary}

\theoremstyle{definition}

\newtheorem{example}[thm]{Example}
\newtheorem{remark1}[thm]{Remark}
\newtheorem{openproblem1}[thm]{Open problem}

\newtheorem{condition}[thm]{Condition}

\newenvironment{ex}{\begin{example}\rm}{\hfill$\Box$\end{example}}

\numberwithin{thm}{section}

\newcounter{FNC}[page]
\def\newfootnote#1{{\addtocounter{FNC}{2}$^\fnsymbol{FNC}$%
     \let\thefootnote\relax\footnotetext{$^\fnsymbol{FNC}$#1}}}

\newcommand{\bs}{\backslash}

\newcommand{\N}{\mathbb{N}}

\newcommand{\R}{\mathbb{R}}

\renewcommand{\P}{\mathbb{P}}

\newcommand{\VV}{\mathcal{V}}
\newcommand{\HH}{\mathcal{H}}

\newcommand{\pol}[1]{{#1}^{\circ}}

\DeclareMathOperator{\conv}{conv}

\DeclareMathOperator{\inter}{int}

\DeclareMathOperator{\qm}{QM}
\DeclareMathOperator{\pos}{pos}

\usepackage{xcolor}

\title[Containment of $\mathcal{H}$-polytopes in $\mathcal{V}$-polytopes]{Sum
of Squares Certificates for Containment of $\mathcal{H}$-polytopes in
$\mathcal{V}$-polytopes}

\author{Kai Kellner}
\author{Thorsten Theobald}

\address{Goethe-Universit\"at, FB 12 -- Institut f\"ur Mathematik,
Postfach 11 19 32, D--60054 Frankfurt am Main, Germany}

\email{\{kellner,theobald\}@math.uni-frankfurt.de}

\date{\today}
\begin{document}

\begin{abstract}
Given an $\HH$-polytope $P$ and a $\VV$-polytope $Q$, the decision problem
whether $P$ is contained in $Q$ is co-NP-complete. This hardness remains if
$P$ is restricted to be a standard cube and $Q$ is restricted to be the
affine image of a cross polytope. While this hardness classification by
Freund and Orlin dates back to 1985, for general dimension there seems 
to be only limited progress on that problem so far.

Based on a formulation of the problem in terms of a bilinear
feasibility problem, we study sum of squares certificates 
to decide the containment problem. These certificates can be
computed by a semidefinite hierarchy.
As a main result, we show that under mild and explicitly known
preconditions the semidefinite hierarchy converges in finitely many steps.
In particular, if $P$ is contained in a large $\mathcal{V}$-polytope $Q$ (in a
well-defined sense), then containment is certified by the first step of the
hierarchy.
\end{abstract}

\maketitle

\section{Introduction}

Convex polytopes (polytopes, for short)
can be represented as the convex hull of finitely many
points (``$\VV$-polytopes'') or as the intersection of finitely many
halfspaces (``$\HH$-polytopes''). For $ a \in \R^k,\ A \in \R^{k\times d} $,
and $ B = [ b^{(1)},\ldots,b^{(l)} ] \in \R^{d\times l} $ let
\[
  P = P_A = \left\{ x\in\R^d\ |\ a - Ax \geq 0 \right\} \text{ and }
  Q = Q_B = \conv( B ) = \conv( b^{(1)},\ldots,b^{(l)} )
\]
be an $\HH$-polytope and a $\VV$-polytope, respectively. The subscript in the
notion of $P$ and $Q$ indicates the dependency on the specific representation
of the polytopes involved. However, if there is no risk of confusion, we often
state $P$ and $Q$ without subscript.
The following two problems are prominent problems in algorithmic polytope
theory (see Kaibel and Pfetsch~\cite{kaibel-pfetsch-2003}).
We always assume that the polytope data is given in terms of rational numbers.

\medskip

\noindent
{\sc Polytope verification:} \\
{\bf Input:} $d \in \N$, an $\HH$-polytope $P \subseteq \R^d$ and a
$\VV$-polytope $Q \subseteq \R^d$. \\
{\bf Task:} Decide whether $P=Q$.

\medskip

\noindent
{\sc Polytope containment} (or {\sc $\HH$-in-$\VV$ containment}): \\
{\bf Input:} $d \in \N$, an $\HH$-polytope $P \subseteq \R^d$ and a
$\VV$-polytope $Q \subseteq \R^d$. \\
{\bf Task:} Decide whether $P \subseteq Q$.

\medskip

While the complexity status of the first problem is open, the second problem
is co-NP-complete (see Freund and Orlin~\cite{freund-orlin-85}); note that it
is trivial to decide the converse question $Q \subseteq P$. It is well known
that the problem of enumerating all facets of a polytope given by a finite set
of points (or, equivalently, enumerating all vertices of a polytope given by a
finite number of halfspaces) can be polynomially reduced to the {\sc Polytope
verification} problem (see Avis et. al.~\cite{abs-97}, Kaibel and
Pfetsch~\cite{kaibel-pfetsch-2003}). Note that enumerating the vertices of an
(unbounded) polyhedron is hard~\cite{kbb-2008}. 
While in fixed dimension, enumeration of the vertices of $P$ can be done in
polynomial time and gives a polynomial time algorithm for {\sc Polytope
containment} (cf.\ Theorem~\ref{co:fixeddim}), progress on approximation
results for the latter problem in general dimension seems to be limited so
far.

In this paper we study the {\sc Polytope Containment} problem. Our main focus
is to consider the problem from the viewpoint of the transition from
linear/polyhedral problems to low-degree semialgebraic problems. To that end,
we formulate the problem as a disjointly constrained bilinear feasibility
problem and consider semialgebraic certificates.
In particular, the reformulation as a bilinear program allows to effectively
apply Putinar's Positivstellensatz~\cite{putinar1993}, which is the primary 
theorem underlying Lasserre's hierarchy~\cite{Lasserre2001} and gives sum of
squares certificates; see also Laurent's extensive
survey~\cite{laurent-survey}. 
The sum of squares can be computed by a hierarchy of semidefinite programs. 

The basic idea of the approach -- which is meanwhile common in polynomial
optimization, but whose understanding of the particular potential for
concrete problems is often challenging -- can be explained as follows.
One point of view towards linear programming is as an application of Farkas'
Lemma which characterizes the (non-)solvability of a system of linear
inequalities. The affine form of Farkas' Lemma~\cite[Corollary
7.1h]{Schrijver1986} characterizes linear polynomials which are nonnegative on
a given polyhedron.
By omitting the linearity condition, one gets a polynomial nonnegativity
question, leading to so called Positivstellens\"atze (or, more precisely,
Nichtnegativstellens\"atze). These Positivstellens\"atze provide a
\emph{certificate} for the positivity of a polynomial function in terms of a
polynomial identity. As in the linear case, the Positivstellens\"atze are the
foundation of polynomial optimization and relaxation methods
(see~\cite{Lasserre2001,lasserre-book, laurent-survey}).

The general machinery from polynomial optimization automatically implies
convergence results, but often these results come with restrictions or
technical assumptions. 

\medskip

\noindent
{\bf Our contributions:}

1. Based on a formulation of the {\sc Polytope containment} problem in terms
  of a bilinear problem (Proposition~\ref{prop:HinV}), we
  characterize geometric properties of a natural bilinear programming
  reformulation; see Corollary~\ref{cor:finsol}.

2. We study the application of sums of squares techniques on the bilinear
  programming formulation.
  An important point is whether the hierarchy always converges in finitely
  many steps. While in the case of strong containment (as defined in
  Section~\ref{se:bilinear}) this property is implied by Putinar's
  Positivstellensatz (Theorem~\ref{thm:archimedean}), in the case of
  non-strong containment this is a critical issue. 
  As a main result of this paper, we show that under mild and explicitly
  known conditions, the semidefinite hierarchy converges in finitely many
  steps (Theorem~\ref{thm:convergence2}) based on results by
  Marshall~(\cite{marshall2008,marshall2009}; see also Nie \cite{nie2012}). 

3. We exhibit structural differences between conventional methods for the
  {\sc Polytope containment} (such as vertex tracking methods) and our
  approach. Theorem~\ref{th:smallinbig} shows that the containment of
  polytopes in ``large'' polytopes (as quantified in the theorem) can already
  be certified in the initial step of the semidefinite hierarchy and thus by
  computing a semidefinite program of polynomial size in the input.

\medskip

While it is a fundamental geometric problem by itself, we mention some
exemplary application scenarios in which the {\sc Polytope containment}
problem occurs.
Generally, many applications in data analysis or shape analysis of point
clouds involve the convex hull of point sets (see, e.g.,
\cite{boehm-kriegel-2001}), and a {\sc Polytope containment} problem can be
used to answer questions about certain (polyhedral) properties on the set.

A specific example is theorem proving in linear real arithmetic: 
A sub-branch in theorem proving is based on formulas in linear real
arithmetic; see, e.g., \cite{cll-2003,monniaux2008}.
Given a set of (quantifier-free) linear inequalities of the form
$L_i(x_1,\ldots,x_d)\ge 0$ in the real variables $x_1, \ldots, x_d$ specifying
the assumptions of a certain theorem, one may ask whether all these solutions
satisfy a certain property $Q$.
If $Q$ is described as the convex hull of a finite number of points, then the
theorem proving problem corresponds to a {\sc Polytope containment} problem.

Let us briefly mention some related problems. Finding the largest simplex in
a $\VV$-polytope is an NP-hard problem~\cite{gkl-95}. However for that problem
Packer has given a polynomial-time approximation~\cite{packer-2004}. Recently,
Gouveia et.\ al. have studied the question which nonnegative matrices are
slack matrices~\cite{ggk-2013}, and they establish equivalence of the decision
problem to the polyhedral verification problem. For containment of polytopes
and spectrahedra see~\cite{Kellner2012,Kellner2013}. Joswig and
Ziegler~\cite{Joswig2004} showed that the {\sc Polytope verification} problem
is polynomially equivalent to a geometric polytope completeness problem.

The paper is structured as follows. After introducing the relevant notation in
Section~\ref{se:prelim}, we study geometric properties of a natural bilinear
programming formulation in Section~\ref{se:bilinear}.  
Section~\ref{sec:putinar} deals with sum of squares certificates for the {\sc
Polytope Containment} problem.
Finally, Section~\ref{se:openquestions} lists several open questions.

\section{Preliminaries\label{se:prelim}}

Recall that a \emph{polyhedron} $P$ is the intersection of finitely many
affine halfspaces in $\R^d$ and a bounded polyhedron is called a
\emph{polytope}~\cite{Ziegler1995}.
Denote by $V(P)$ the set of vertices of a polytope $P$, and by $F(P)$ the set
of facets. By McMullen's Upper bound Theorem~\cite{mcmullen70}, any
$d$-polytope with $k$ vertices (resp.\ facets) has at most
\[
  \binom{k-\left\lfloor\frac{1}{2}(d+1)\right\rfloor}{k-d} 
  + \binom{k-\left\lfloor\frac{1}{2}(d+2)\right\rfloor}{k-d}
\]
facets (resp.\ vertices).
This bound is sharp for neighborly polytopes such as cyclic polytopes.

Our model of computation is the binary Turing machine:
polytopes are given in terms of rational numbers, and the size of the input
is defined as the length of the binary encoding of the input data (see, e.g.,
\cite{gritzmann-klee-93}).
It is well-known that the complexity of deciding containment of one polytope
in another one strongly depends on the type of input representations.
In particular, the following hardness statement is known.

\begin{prop}[{\cite{freund-orlin-85,gritzmann-klee-93}}]
\label{prop:complexity}
The {\sc Polytope Containment} problem is co-NP-complete. 

This hardness remains if $P$ is restricted to be a standard cube and $Q$ is
restricted to be the affine image of a cross polytope.

If the dimension is fixed, then the problem of deciding whether an
$\HH$-polytope is contained in a $\VV$-polytope can be decided in polynomial
time.
\end{prop}

Note that the result for fixed dimension can be slightly strengthened.

\begin{cor}\label{co:fixeddim}
If the dimension of $P$ or the dimension of $Q$ is fixed, then containment of
an $\HH$-polytope $P$ in a $\VV$-polytope $Q$ can be decided in polynomial
time.
\end{cor}

\begin{proof}
If the dimension of $P$ is fixed, then first compute the affine hull of $P$. 
This can be done in polynomial time. Taking that affine hull as ambient space
in fixed dimension, $P$ can be transformed into a $\VV$-representation in
polynomial time. It remains to decide containment of a $\VV$-polytope in a
$\VV$-polytope, which can be done in polynomial time.

Similarly, if the dimension of $Q$ is fixed, an $\HH$-representation of $Q$
can be computed in polynomial time, and the resulting problem of deciding
whether an $\HH$-polytope is contained in an $\HH$-polytope can be decided in 
polynomial time.
\end{proof}

Throughout the paper, we assume boundedness and nonemptyness of $P$ as
well as $0 \in \inter Q$,
where $\inter Q$ denotes the interior of $Q$.
All these properties can be tested in polynomial time~\cite{khachiyan1980}.
If $Q$ is full-dimensional and $0 \not\in \inter Q$, one can translate $Q$
and $P$ by the centroid of the vertices of $Q$. 
Recall that the polar polyhedron of $Q$ is 
\[
  \pol{Q} = \left\{ z \in \R^d\ |\ \mathds{1}_l - B^T z \geq 0 \right\} ,
\]
where $\mathds{1}_l$ denotes the all-1-vector in $\R^l$. Since we assume
$0 \in \inter Q$, $\pol{Q}$ is a polytope (i.e., bounded) and
$Q^{\circ\circ}=Q$.

\section{A bilinear approach to the
{\sc Polytope containment} problem\label{se:bilinear}}

We first collect some geometric properties of the {\sc Polytope containment}
problem. Our starting point is the following reformulation of the 
{\sc Polytope containment} problem as a bilinear problem.

\begin{prop} \label{prop:HinV}
Let the $\HH$-polytope $P = \left\{ x\in\R^d\ |\ a - Ax \geq 0 \right\}$
be nonempty and the $\VV$-polytope $Q=\conv(B)=\conv(b^{(1)},\ldots,b^{(l)})$
containing the origin in its interior.
\begin{enumerate}
  \item 
  $P$ is contained in $Q$ if and only if 
  $x^T z \le 1$ for all $(x,z) \in P \times Q^{\circ}$.
  That is, $P \subseteq Q$ if and only if the maximum
\begin{equation} \label{eq:hinv}
  \mu^* := \max \{ x^T z \, | \, (x,z) \in P \times\pol{Q} \} 
\end{equation}
is at most~1. 
  \item We have $\mu^* = 1$
  if and only if $P\subseteq Q$ and $\partial P\cap\partial Q\neq\emptyset$.
\end{enumerate}
\end{prop}

Motivated by the second statement, we say that $P$ is \emph{strongly contained}
in $Q$ if $P \subseteq Q$ and $\partial P\cap\partial Q =\emptyset$.
Since $Q$ is full-dimensional, this is equivalent to $P \subseteq\inter Q$.
Note that strong containment differs from (set-theoretic) strict inclusion as
the latter allows common boundary points while strong containment does not.

\begin{proof}
To (1):
If $P\subseteq Q$ then for any $x \in P$ we have $x^T z\leq 1$ for all
$z\in\pol{Q}$. Conversely, if $x^T z\leq 1$ holds for all $z\in\pol{Q}$, then
for any $x\in P$ we have $x\in Q^{\circ\circ} = Q$. 

To (2):
Let $P\subseteq Q$ and $\partial P \cap\partial Q$ be nonempty. Then there
exists a vertex $v\in V(P)$ and a facet $F \in F(Q)$ such that $v\in F$.
Since $0 \in \inter Q$, $F$ defines a vertex $f$ of the polar $Q^{\circ}$.
Further $f^T v = 1$ implies that the maximum is at least one. By part (1) of the
statement, the maximum must be exactly one.

Conversely, if the maximum is one, then $x^T z\leq 1$ for all 
$(x,z)\in P\times\pol{Q}$. Therefore, since the set $P\times\pol{Q}$ is 
compact, there exists a point 
$(\bar{x},\bar{z})\in P\times\pol{Q}$ such that $\bar{x}^T \bar{z} = 1$.
Hence $\bar{x}^T z\leq 1$ for all $z\in\pol{Q}$ and $\bar{x}^T \bar{z} = 1$,
i.e., $\bar{x}$ defines a supporting hyperplane of $\pol{Q}$. Thus $\bar{x}$
is a boundary point of $Q$. Similarly, $x^T \bar{z}\leq 1$ for all $x\in P$
and $\bar{x}^T \bar{z} = 1$, implying $\bar{x}\in\partial P$.
Consequently, $\bar{x} \in \partial Q \cap\partial P $.
\end{proof}

The following characterization of the optimal solutions to~\eqref{eq:hinv} is
a slight extension of a result by Konno~\cite{Konno1976} on bilinear
programming.

\begin{prop} \label{prop:bilin1}
Let $\inter P \neq \emptyset$ and $0 \in \inter Q$.
Then the set of optimal solutions to~\eqref{eq:hinv} is a set of proper faces
$F\times G$ of $P\times\pol{Q}$, and the maximum is attained at a
pair of vertices of $P$ and $\pol{Q}$. 
\end{prop}

For the convenience of the reader, we recall the short proof.

\begin{proof}
Let $ (\bar{x},\bar{z}) \in P \times\pol{Q} $ be an optimal solution.
Then the set $G = \{z \in \pol{Q} \, | \, \bar{x}^T z = \bar{x}^T \bar{z} \}$
is a non-empty face of $\pol{Q}$. 
For all $\hat{z}\in G$, let $F$ be the set of maximizers of
$\max \{ x^T \hat{z} \, | \, x \in P\}$.
Consequently, for  $(\hat{x},\hat{z})\in F\times G$ we have
$ \hat{x}^T \hat{z} = \bar{x}^T \hat{z} = \bar{x}^T \bar{z}$
and, by the optimality of
$(\bar{x},\bar{z})$, $F\times G$ is contained in the set of optimal solutions.
To complete the proof of the first
part of the statement, note that every boundary point of a polytope 
has a unique minimal face containing it.

Since the set of optimal solutions is a set of proper faces of 
$P \times Q^{\circ}$, 
there exists an optimal pair of vertices of $P$ and $\pol{Q}$.
\end{proof}

There is a nice geometric interpretation of the latter proposition. Since, in
the case $0 \in\inter Q$, each vertex of $\pol{Q}$ corresponds to a facet of
$Q$ and vice versa, an optimal solution $(x,z)\in V(P)\times V(\pol{Q})$
of~\eqref{eq:hinv} yields a pair of a vertex of $P$ and a facet defining
normal vector of $Q$.
However, since computing the set of vertices $V(\pol{Q})$ is an NP-hard
problem, it is not reasonable to reduce the problem to the set of vertices in
general. 

The optimal value of problem~\eqref{eq:hinv} might be attained by other
boundary points than vertices and, moreover, there might be infinitely many
optimal solutions. From a geometric point of view, this only occurs in
somewhat degenerate cases.

\begin{cor} \label{cor:finsol}
Let $\inter P \neq \emptyset$ and $0 \in \inter Q$.
Problem~\eqref{eq:hinv} has finitely many optimal solutions if and only if
every optimal solution of Problem~\eqref{eq:hinv} is a pair of vertices of $P$
and $\pol{Q}$.
\end{cor}

As there always exists a pair of vertices that is an optimal solution
to~\eqref{eq:hinv}, it is natural to ask for vertex tracking algorithms.
This is the approach in, e.g.,~\cite{Gallo1977,Konno1976}.
So far, no converging algorithm is known based on this approach. 
Note that the formulation as a bilinear programming problem from
Proposition~\ref{prop:HinV} also allows to apply existing nonlinear 
programming techniques for non-convex quadratic optimization (see, e.g., 
\cite{chen-burer-2012}).
In the next section, we study the bilinear reformulation of the {\sc
Polytope containment} from the viewpoint of algebraic certificates and
semidefinite relaxations yielding a generically convergent algorithm.

\section{Sum of squares certificates}
\label{sec:putinar}

In this section, we study sum of squares techniques for the {\sc Polytope
containment} problem. Our main goal is to show that in the situation of
Corollary~\ref{cor:finsol}, the corresponding semidefinite hierarchy yields 
a certificate for containment after finitely many steps; see
Theorems~\ref{thm:archimedean} and~\ref{thm:convergence2}.

\subsection{Putinar's Positivstellensatz}

Consider a set of polynomials $ G = \{ g_1,\ldots,g_k \} \subseteq \R[x] $ in the
variables $ x = (x_1,\ldots,x_d) $. The quadratic module generated by $G$ is
defined as 
\[
  \qm(G) = 
  \left\{ \sigma_0 + \sum_{i=1}^k \sigma_i g_i \ 
  |\ \sigma_i \in\Sigma[x] \right\} ,
\]
where $ \Sigma[x] \subseteq \R[x] $ is the set of sum of squares polynomials.
Here, a polynomial $ p\in\R[x] $ is called \emph{sum of squares} (sos) if it
can be written in the form $ p = \sum_i h_i(x)^2 $ for some $h_i \in\R[x]$.
Equivalently, $p$ has the form $ [x]^T Q [x] $, where $[x]$ is the vector
of all monomials in $x$ up to half the degree of $p$ and $Q$ is a positive
semidefinite matrix of appropriate size. Checking whether a polynomial is sos
is a semidefinite feasibility problem.

Obviously, every element in $\qm(G)$ is nonnegative on the semialgebraic set
$S = \{x\in\R^d\ |\ g(x)\geq 0\ \forall g\in G\}$. In~\cite{putinar1993}
Putinar showed that for positive polynomials the converse is true under some
regularity assumption.

A quadratic module $\qm(G)$ is called \emph{Archimedean} if there is a
polynomial $ p \in \qm(G) $ such that the level set 
$ \{x\in\R^d\ |\ p(x)\geq 0 \} $ is compact, or, equivalently, the polynomial
$N-(x_1^2+\dots+x_d^2) \in\qm(G)$ for some positive integer $N$; see
Marshall's book~\cite{marshall2008} for more equivalent characterizations.

\begin{prop}[{Putinar's Positivstellensatz~\cite{putinar1993}. See
also~\cite[Theorem 5.6.1]{marshall2008}}] \label{prop:putinar}
Let $S=\{x\in\R^d\ |\ g(x)\geq 0\ \forall g\in G\}$ for some finite
subset $G\subseteq\R[x]$. If the quadratic module $\qm(G)$ is Archimedean,
then $\qm(G)$ contains every polynomial $f\in\R[x]$ positive on $S$.
\end{prop}

The Archimedean condition in the proposition is not very restrictive.
Especially, in our case of interest where all polynomials $g_i$ are linear and
$S$ is compact, the condition is always fulfilled; 
see~\cite[Theorem 7.1.3]{marshall2008}. 

In order to apply Putinar's Positivstellensatz to polynomial optimization,
consider an optimization problem
\begin{align} \label{eq:pop}
  \sup\left\{ f(x)\ |\ g_i(x) \geq 0\, , \; i=1,\ldots, k \right\}
\end{align}
with $f,g_1,\ldots,g_k \in\R[x] $. Clearly, this is the same as to find
the infimum of $\mu$ such that $ \mu - f(x) \geq 0$ on the set $S$. A common way to
tackle the latter problem is to replace the nonnegativity condition by an sos
condition. This is a semi-infinite program since deciding membership can be
rephrased as a semi-infinite feasibility problem. In order to get a
(finite-dimensional) semidefinite program, we truncate the quadratic
module $\qm(G)$ by considering only monomials up to a certain degree $2t$,
\[
  \qm_t(G) = \left\{ \sigma_0 + \sum_{i=1}^k \sigma_i g_i \ 
    |\ \sigma_i \in\Sigma[x] \text{ with } \deg(\sigma_0)\leq 2t 
    \text{ and } \deg(\sigma_i g_i) \leq 2t \right\} .
\]
The $t$-th sos program has the form
\begin{align} \label{eq:sos}
  \mu(t) = \inf \left\{ \mu\ |\ \mu - f(x) \in\qm_t(G) \right\}.
\end{align}
Clearly, the sequence of \emph{truncated quadratic modules} is increasing with
respect to inclusion as $t$ grows. Thus the sequence of optimal values
$\mu(t)$ is monotone decreasing and bounded from below by the optimal
value of~\eqref{eq:pop}.
Generally, the infimum is not attained in~\eqref{eq:sos}.

The dual problem to~\eqref{eq:sos} can be formulated in terms of moment
matrices, again leading to an SDP relaxation of the polynomial optimization
problem~\eqref{eq:pop}. From a computational point of view it is often easier
(i.e., faster) to compute the dual side. This is because of the time consuming
process of extracting coefficients in a formal sos representation.
It is known that there is no duality gap between the primal and dual problem,
whenever the quadratic module is Archimedean and $S$ contains an interior
point~\cite[Theorem 5.21]{lasserre-book}.
Since we do not use the dual side here, we refer interested readers to
Lasserre's fundamental work~\cite{Lasserre2001}.

\subsection{Sum of squares certificates for {\sc Polytope containment}}
To keep notation simple, we denote the (truncated) quadratic module generated
by the linear constraints $a-Ax$ and $1-B^T z$ by $\qm_t(A,B)$.
The sos formulation of problem~\eqref{eq:hinv} reads as
\begin{align}
\begin{split} \label{eq:putinar}
  \mu(t) 
  &= \inf \left\{ \mu\ |\ \mu -x^T z \in \qm_t(A,B) \right\} .\\
\end{split}
\end{align}

Denote the $i$-th constraint defining $P\times\pol{Q}$ by $g_i$. Let 
$ \mu - x^T z = \sigma_0 + \sum_{i=1}^{k+l} \sigma_i g_i $ be an sos
representation. Assume $t=1$. Then monomials of degree at most 2 appear, i.e.,
$\deg(\sigma_0) \in\{0,2\}$ and $\deg(\sigma_i g_i) \leq 2$. Since 
$\deg(g_i) = 1$ and $\sigma_i$ is sos, $\sigma_i$ must be constant (otherwise
monomials of degree greater than 2 appear).
Thus $ \deg(\sum_i \sigma_i g_i) \leq 1$.
Moreover, if $\deg(\sigma_0) = 2$, then purely quadratic terms like $x_j^2$ or 
$z_j^2$ appear for some $j$ on the right-hand side while the coefficients of
these terms are zero on the left-hand side. 
As a consequence, the first order of the hierarchy making sense is
$t=2$. We call $t=2$ the \emph{initial step of the hierarchy}.

Asymptotic convergence of the hierarchy in the general case and finite
convergence in the strong containment case follow easily from the general
theory. 

\begin{thm} \label{thm:archimedean}
Let $P$ be a nonempty 
$\HH$-polytope and $Q$ be a $\VV$-polytope with $0\in\inter Q$.
\begin{enumerate}
  \item
  If $\mu(t) \leq 1 $ for some integer $ t\geq 2$, then $P\subseteq Q$.
  \item
  The hierarchy~\eqref{eq:putinar} converges asymptotically from above to the
  optimal value $\mu^*$ of problem~\eqref{eq:hinv}.
  \item 
  If $P$ is strongly contained in $Q$, then the hierarchy~\eqref{eq:putinar}
  decides the \textsc{Polytope Containment} problem in finitely many steps.
\end{enumerate}
\end{thm}

\begin{proof}
The first statement is clear by construction of the hierarchy.

Consider the second statement.
Since all constraints are linear in $x,z$ and the feasible region is bounded,
the quadratic module generated by the constraints of problem~\eqref{eq:hinv}
is Archimedean~\cite[Theorem 7.1.3]{marshall2008} and thus contains all 
polynomials $f(x,z)\in\R[x,z]$ positive on $P\times\pol{Q}$ by
Putinar's Positivstellensatz~\ref{prop:putinar}. 
Let $\mu^*$ be the optimal value of problem~\eqref{eq:hinv}.
Then $\mu^* - x^T z \ge 0$ on $P\times\pol{Q}$ and hence 
$\mu^* +\epsilon - x^T z \in\qm(A,B)$ for all $\epsilon>0$.

If $P$ is strongly contained in $Q$, then $\mu^*<1$ and thus $1-x^T
z\in\qm(A,B)$.
\end{proof}

A priori it is not clear whether in the non-strong case finite convergence
holds.
In fact, for general polynomials, there are examples where finite
convergence is not possible.
As our main result, we provide a partial extension of
Theorem~\ref{thm:archimedean} to the case where the bilinear optimization
problem~\eqref{eq:hinv} has only finitely many optimal solutions (as
characterized in Corollary~\ref{cor:finsol}).

\begin{thm} \label{thm:convergence2}
Let $P = \left\{x\in\R^d\ |\ a-Ax\geq 0 \right\} $ be a 
full-dimensional $\HH$-polytope and let $Q = \conv(B)$ be a $\VV$-polytope containing the
origin in its interior. Assume that one of the equivalent statements in
Corollary~\ref{cor:finsol} holds. Then $\mu^* - x^T z \in\qm(A,B)$, and
thus the hierarchy~\eqref{eq:putinar} decides the \textsc{Polytope
Containment} problem in finitely many steps.
\end{thm}

To prepare for the proof, we show that in the semidefinite hierarchy for the
{\sc Polytope containment} problem, the sos formulation is invariant under
redundant constraints, i.e., redundant inequalities in the
$\HH$-representation of $P$ or redundant points in the $\VV$-representation of
$Q$.
Note that for a general semialgebraic constraint set this is not always
true, even in the case of optimizing a linear function over it;
see~\cite[Section 5.2]{Henrion2008} for a well-known example (cf.\ also
\cite{gva-2011}).
Recall that every $\HH$-representation of a certain polytope contains the
facet defining halfspaces. Similarly, the vertices are part of each
$\VV$-representation.

\begin{lemma}[Redundant constraints] \label{lem:redundant}
Let $ P_A = \left\{ x\in\R^d\ |\ a-Ax \geq 0 \right\} $ and $Q_B = \conv(B)$
be nonempty polytopes with $ a\in\R^{k+1},\ A\in\R^{(k+1)\times d}$, and
$B\in\R^{d\times (l+1)}$.
\begin{enumerate}
  \item 
  If $(a-Ax)_{k+1}\geq 0$ is a redundant inequality in the
  $\HH$-representation of $P_A$, then it is also redundant in the sos
  representation~\eqref{eq:putinar}, i.e., the inclusion $P_A\subseteq Q_B$
  is certified by a certain step of the hierarchy if and only if 
  $P_{A\bs A_{k+1}}\subseteq Q_B$ is certified by the same step.
  \item 
  If $b^{(l+1)}$ is a redundant point in the $\VV$-representation of $Q_B$,
  then it is also redundant in the sos representation~\eqref{eq:putinar}, i.e., 
  $P_A \subseteq Q_B$ is certified by a certain step of the hierarchy if and
  only if $P_A\subseteq Q_{B\bs b^{(l+1)}}$ is certified by the same step.
\end{enumerate}
\end{lemma}

\begin{proof}
We only prove statement (1), the proof of part (2) is analog. Consider an
sos representation of $\mu(t)-x^T z$ for some $t\geq 2$, 
\[
  \mu(t) -x^T z = \sigma_0 + \sum_{i=1}^{k+1} \sigma_i \left(a-Ax\right)_i
  + \sum_{i=1}^l \sigma_{k+1+i} \left(1-B^T z\right)_i
  \in \qm(A,B) \, ,
\]
where $\sigma_0,\ldots,\sigma_{k+l+1}\in\Sigma[x,z]$ are sos polynomials with
$\deg\sigma_0 \le 2t$ and $\deg\sigma_i \le 2t-2$ for $i\in\{1,\ldots,k+l+1\}$.
Since $(a-Ax)_{k+1}$ is redundant in the description of $P_{A}$, we can write
it as a conic combination of the remaining linear polynomials,
\[
  (a-Ax)_{k+1} = \lambda_0 + \lambda^T(a-Ax),
  \ \lambda\in\R^k_+,\ \lambda_0\in\R_+.
\]
Replacing $\sigma_{k+1} (a-Ax)_{k+1}$ in the sos representation yields
\[
  \mu(t)-x^T z = \sigma'_0 + \sum_{i=1}^k \sigma'_i \left(a-Ax\right)_i
  + \sum_{i=1}^l \sigma_{k+1+i} \left(1-B^T z\right)_i 
  \in\qm(A\bs A_{k+1},B) ,
\]
where $\sigma'_i=\lambda_i\sigma_{k+1}+\sigma_i\in\Sigma[x,z]$ with degree
$\deg(\sigma'_i)=\max\{\deg(\lambda_i\sigma_{k+1}),\deg(\sigma_i)\}\leq 2t-2$
for $i \in \{0,\ldots, k\}$.
\end{proof}

To prove Theorem~\ref{thm:convergence2}, we introduce a sufficient convergence
condition by Marshall (see \cite{marshall2008,marshall2009}) which is based on
a boundary Hessian condition.

Given $g_1,\ldots,g_k \in\R[x]$ and a boundary point $\bar{x}$ of
$S = \{x \in\R^d\ |\ g_i(x)\ge 0,\ i=1,\ldots,k\}$.
We assume that (say, by an application of the inverse function theorem) there
exists a local parameterization for $\bar{x}$ in the following sense:
There exist open sets $U,V \subseteq\R^{d}$ such that $\bar{x} \in U$,  
$\phi : U \to V,\ x \mapsto t := (t_1,\ldots,t_d) $ is bijective, the inverse
$\phi^{-1} : V \to U$ is a continuously differentiable function on $V$,
and for some $r \in\{1,\ldots,d\}$ let $t_1 = g_1, \ldots, t_r = g_r$ on $U$.

\begin{condition}[Boundary Hessian condition, BHC] \label{con:hessian}
Given a polynomial $f\in\R[x]$, denote by $f_1$ and $f_2$ the linear and
quadratic part of $f$ in the localizing parameters $t_1,\ldots,t_d$,
respectively.
Let $R = \{ (t_1,\ldots,t_d)\in\R^d \, | \, t_1 \ge 0,\ldots,t_r \ge 0 \}$.
If the linear form $ f_1 = c_1 t_1 + \dots + c_r t_r $ has only negative
coefficients and the quadratic form $ f_2(0,\ldots,0,t_{r+1},\ldots,t_d) $ is
negative definite, then the restriction $f_{|R}$ has a local maximum at
$\bar{x}$.
\end{condition}

Using this condition, the following generalization of Putinar's Theorem can be
stated.

\begin{prop}[{\cite[Theorem 9.5.3]{marshall2008}, see also
\cite[Theorem 3.1.7]{scheiderer-guide}}]
\label{prop:marshall}
Let $ f,g_1,\ldots,g_k \in\R[x]$, and suppose that the quadratic module
$\qm(G)$ generated by $G=\{g_1,\ldots,g_k\}$ is Archimedean. Further assume
that for each global maximizer $\bar{x}$ of $f$ over 
$S = \{x \in\R^d\ |\ g(x)\ge 0\ \forall g\in G\}$ there exists an index
set $I\subseteq\{1,\ldots,d\}$ such that (after renaming the variables w.r.t.\
the indices in $I$ and w.r.t.\ the indices not in $I$) $f$ satisfies BHC at
$\bar{x}$. Denote by $f_{\max}$ the global maximum of $f$ on $S$.
In this situation, $f_{\max}-f\in\qm(G)$.
\end{prop}

Our goal is to show that under the assumptions of
Theorem~\ref{thm:convergence2} the boundary Hessian condition holds.
We will use the following version of the Karush-Kuhn-Tucker conditions adapted
to the bilinear situation.

\begin{lemma}[{\cite[Section~5.1]{bss-2006}}] \label{lem:multiplier}
Let $f(x,z) \in\R[x,z]$ be a continuously differentiable function and let 
$ \P := P_A \times P_B = \{ (x,z)\in\R^{2d}\ |\ a-Ax \geq 0,\ b-Bz \geq 0\} $
be the product of two nonempty polytopes. If $f$ attains a local maximum at 
$(\bar{x},\bar{z})$ on $\P$, then there exists $(\alpha,\beta)$ such that 
\begin{align}
\begin{split} \label{eq:multi2} 
  \nabla f(\bar{x},\bar{z}) &= 
  \begin{bmatrix} A^T & 0 \\ 0& B^T \end{bmatrix} 
  \begin{pmatrix} \alpha \\ \beta \end{pmatrix} \\
  0 &=\alpha_i (a - A\bar{x})_i = \beta_j (b - B \bar{z})_j \, , \quad 
  i=1,\ldots,k, \; j=1,\ldots, l \\
  &\alpha \geq 0,\ \beta \geq 0 .
\end{split}
\end{align}
\end{lemma}

In the lemma, only multipliers corresponding to active constraints can be
positive, since otherwise one of the equations~\eqref{eq:multi2} is violated.

We are now able to prove Theorem~\ref{thm:convergence2}. In a more general
setting, Nie used the Karush-Kuhn-Tucker optimality conditions to certify the
BHC; see~\cite{nie2012}. Because of the special structure of
problem~\eqref{eq:hinv}, we do not need the whole machinery used by Nie. In
particular, the local parameterization needed for the BHC (see the paragraph
before Condition~\ref{con:hessian}) comes from an affine variable
transformation. As a consequence, for {\sc Polytope containment}, our direct
approach allows to prove a stronger result than we would obtain just by
applying Nie's Theorem. Specifically, we obtain a geometric characterization
of the degenerate situations as given in Theorem~\ref{thm:convergence2}.

\begin{proof}[Proof (of Theorem~\ref{thm:convergence2})]
Let $ (\bar{x},\bar{z}) \in P\times Q^{\circ}$ be an arbitrary but fixed
optimal solution. By Lemma~\ref{lem:multiplier} there exists
$(\alpha,\beta)\in\R^{k+l}$ such that 
\begin{align}
\begin{split} \label{eq:proof1}
  (\bar{z},\bar{x}) &= ( A^T \alpha , B \beta ) \\
    0 &=\alpha_i (a - A\bar{x})_i = \beta_j (\mathds{1} - B^T \bar{z})_j \, , \quad 
  i=1,\ldots,k, \; j=1,\ldots, l\\
  & \alpha \geq 0,\ \beta \geq 0 .
\end{split}
\end{align}
As mentioned before, only multipliers corresponding to active constraints can
be positive.
Denote by $\mathbb{I}$ the collection of index sets of active constraints in
$\bar{x}$ with positive multipliers in~\eqref{eq:proof1}.
Assume $|I|<d$ holds for any such index set $I\in\mathbb{I}$.
Then $\bar{z}$ is a positive combination of at most $d-1$ active constraints in
$\bar{x}$. That is, $\bar{z}$ does not lie in the interior of the outer
normal cone of the vertex $\bar{x}$.
Equivalently, $\bar{z}$ lies in the outer normal cone of an at least
one-dimensional face $F$ of $P$ containing $\bar{x}$.
Then $x^T\bar{z}=\bar{x}^T \bar{z}$ for all $x\in F$, in contradiction to the
assumption of the theorem and Corollary~\ref{cor:finsol}.
Thus there must be an index set $\bar{I}\in\mathbb{I}$ of cardinality of at
least $d$.
By a symmetric argument, there exists an index set $\bar{J}$ of 
active
constraints in $\bar{z}$ with positive multipliers in~\eqref{eq:proof1} that
has cardinality $|\bar{J}|\ge d$.

If existent, we pick such index sets $I$ and $J$ with $|I|=|J|=d$. 
Otherwise, we proceed as follows, where $A_i$ denotes the $i$-th row of $A$.
As $\bar{x}$ is a vertex of $P$, the cone
$\pos\{A^T_i \, | \, i \in I\}$ is full-dimensional and 
contains $\bar{z}$ in its interior.
There exist linearly independent
$v^{(1)},\ldots,v^{(d)} \in \pos\{A^T_i \, | \, i \in I\}$
generating a simplicial subcone with
$\bar{z} \in \inter \pos \{v^{(1)},\ldots,v^{(d)} \} 
  = \big\{ \sum_{i=1}^d \mu_i v^{(i)} \, | \, \mu_i > 0 \big\} \, .
$

Indeed, introducing the vectors $v^{(i)}$ corresponds to adding
redundant inequalities to the $\HH$-polytope $P$ which are active in $\bar{x}$.
By Lemma~\ref{lem:redundant},
if we show the statement for this redundant
representation of $P$, it is also applicable to the original set.
Thus, after possibly introducing these redundancies, 
there exists an index set $I$, $|I| = d$, of linearly
independent active constraints with positive coefficients. And analogously
for the subset $J$ and the representation of $\pol{Q}$.

We apply the affine variable transformation $\phi:\ \R^{2d}\to\R^{2d}$ defined
by
\[
  \phi(x,z) = \begin{bmatrix} (a-Ax)_I \\ (\mathds{1}_l-B^T z)_J \end{bmatrix}
\]
and denote the new variables by
$ (s,t) := (s_1,\ldots,s_d,t_1,\ldots,t_d) =
  (\phi_{1}(x,z),\ldots,\phi_{2d}(x,z)) $.
Clearly, $\phi$ is a local parameterization at
$(\bar{x},\bar{z})$ in the sense of Condition~\ref{con:hessian}. The inverse
of $\phi$ is given by 
\[
  (s,t)\mapsto \begin{bmatrix}
                 A_I^{-1}(a_I -s) \\
		 (B^T_J)^{-1}(\mathds{1}_J -t)
               \end{bmatrix} .
\]
Setting $M:=B_J^{-1}A_I^{-1}$, the objective $x^T z$ has the form
\[
  f(s,t) := \left(A_I^{-1}(a_I -s)\right)^T 
  \left((B^T_J)^{-1}(\mathds{1}_J -t)\right)
  = s^T M^T t -s^T M^T \mathds{1}_J - a_I^T M^T t + a_I^T M^T \mathds{1}_J 
\]
in the local parameterization space. Denote by $f_1$ the homogeneous part of
degree 1. Then 
$(\bar{x},\bar{z}) =\phi^{-1}(0) =(A_I^{-1}a_I , (B_J^T)^{-1}\mathds{1}_J )$
implies 
\[
  \nabla_{s,t} f_1 (0) 
  = (-\mathds{1}_J^T B_J^{-1}A_I^{-1}, -a_I^T (A_I^T)^{-1} (B_J^T)^{-1} ) 
  = (-\bar{z}^T A_I^{-1}, -\bar{x}^T (B_J^T)^{-1}) 
  = (-\alpha_I^T, -\beta_J^T)\, ,
\]
where the last equation follows from the first identity in~\eqref{eq:proof1}.
Thus the first part of Condition~\ref{con:hessian} is satisfied.
Since $ |I|+|J| =r= 2d $ (where $r$ is from Condition~\ref{con:hessian}), the
second assumption in Condition~\ref{con:hessian} is obsolete. Therefore, by
Proposition~\ref{prop:marshall}, $ \mu^* - x^T z \in\qm(A,B)$.
\end{proof}

Geometrically, the proof uses that in a global maximizer of the bilinear
problem (which by Corollary~\ref{cor:finsol} is a vertex of the polytope
$P\times Q^{\circ}$) traversing along one of the outgoing edges strictly
decreases the objective function, making the second assumption in
Condition~\ref{con:hessian} obsolete.

\subsection{A sufficient criterion and examples} 
\label{sec:examples}

To illustrate the behavior of the approach, we discuss some properties and
a sufficient criterion. It is helpful to start from the following two
structured examples.

\begin{ex} \label{ex:cubecross}
Let $P$ be the cube 
$P = \{ -1 \leq x_i \leq 1,\ i=1,\ldots,d \} \subseteq \R^d$,
and let $ \pol{Q} = \{ -1 \leq e z_i \leq 1,\ i=1,\ldots,d \} \subseteq \R^d$, 
i.e., $Q$ is
a $d$-dimensional cross polytope scaled by a positive integer $e$. Clearly,
$P\subseteq Q$ if and only if $e \geq d$. 

Consider the sos representation of order $t=2$
\begin{align*}
  \frac{d}{e} - x^T z &= 
  \frac{1}{8e} \sum_{i=1}^d \left[ (1-x_i) [(1+x_i)^2 + (1+ez_i)^2]
  + (1+x_i) [ (1-x_i)^2 + (1-ez_i)^2 ] \right] \\
  \ &+ \frac{1}{8e} \sum_{i=1}^d \left[ (1-ez_i) [( 1+x_i)^2 + (1+ez_i)^2] 
  + (1+ez_i) [(1-x_i)^2 + (1-ez_i)^2] \right] .
\end{align*}
If $ e \geq d $, then $ 1-x^T z \geq \frac{d}{e} - x^T z \geq 0 $, certifying
the containment $ P\subseteq Q $ (with strongness if $e > d$). If $ e < d $,
then $ 1-x^T z < \frac{d}{e} - x^T z $. This is not a certificate for
non-containment, since there might be a different sos representation. However,
in this case this is not possible since $e\geq d$ is a necessary condition for
containment.
Note that the necessary order is low and the number of terms is linear in the
dimension.
\end{ex}

\begin{ex} \label{ex:cube}
Let $P$ be the $d$-dimensional cube in $\HH$-representation as in
Example~\ref{ex:cubecross} and $Q= \conv(\{-1,1\}^d)$ be the $d$-dimensional
cube in $\VV$-representation.
Denote by $rP:=\{ x\in\R^{d}\ |\ -r\le x_i \le r,\ i=1,\ldots,d\}$ the
$r$-scaled cube with edge length $2r$. Clearly, $rP\subseteq Q$ if and only if
$0\leq r\leq 1$. This containment problem is combinatorially hard since the
number of inequalities is equal to $2d + 2^d $ and thus exponential in the
dimension.

We are interested in the maximal $r$ such that the containment $rP\subseteq Q$
is certified by a certain step $t$. We also ask for the minimal $t$ such that
$P= 1P \subseteq Q$ is certified. Note that for $r=1$ a priori the existence
of such a $t$ is not clear since neither Theorem~\ref{thm:archimedean} nor
Theorem~\ref{thm:convergence2} applies.

\begin{table}
\begin{tabular}{c|cccc}
  \toprule
  $d\, \bs\, t $ & 2 & 3 & 4 & 5 \\
 \midrule
 2 & 0.7071 & 0.9937 & 0.9994 & 0.9999 \\
 3 & 0.5774 & 0.8819 & 0.9949 & 0.9994 \\
 4 & 0.5000 & 0.7906 & 0.9461 &  \\ 
 5 & 0.4472 & 0.7211 &  &  \\
  \bottomrule
\end{tabular}
\\[+0.5ex]
\caption{Computational test of containment of an $r$-scaled $\HH$-cube in
a $\VV$-cube as described in Example~\ref{ex:cube}. The entries denote the
maximal $r$ (rounded to four decimal places) such that containment in
dimension $d$ is certified by the order $t$.}
\label{tab:numeric-cube}
\end{table}

For $r= \sqrt{d}/d$, $rP\subseteq Q$ is certified by the sos representation
\begin{align*}
  1 - x^T z &= 
  \frac{1}{2} \sum_{i=1}^d \left(x_i-z_i\right)^2 
  + \frac{1}{2^{d+1}}\sum_{v\in\{-1,1\}^d}\left(1+v^T z\right)^2 \left(1-v^T
z\right) \\
  &+ \frac{1}{4r}\sum_{i=1}^d \left( \left(r+x_i\right)^2
  \left(r-x_i\right) +\left(r-x_i\right)^2 \left(r+x_i\right) \right) .
\end{align*}
We are not aware of a more compact sos representation. Numerically, for
$t=2$ and $d\leq 5$, we get $r(d) = \sqrt{d}/d$; see
Table~\ref{tab:numeric-cube}.

Note that the variable transformations $x_i \mapsto \frac{1}{\lambda_i} x_i'$
and $z_i \mapsto \lambda_i z_i'$ give a certificate for the containment of the
box $[-\lambda_1,\lambda_1] \times \cdots \times [-\lambda_d,\lambda_d]$ in 
the box $\frac{d}{\sqrt{d}}Q$, where
$ Q= \conv(\{-\lambda_1,\lambda_1\} \times \cdots \times
  \{-\lambda_d,\lambda_d\})$.
\end{ex}

The consideration of the box leads to the following sufficient
criterion for the existence of a certificate in the 
initial relaxation step. The criterion implies that the containment
of any polytope within any other ``sufficiently large'' polytope 
is certified already in the initial relaxation step.

\begin{thm}\label{th:smallinbig}
Let $P$ be an $\mathcal{H}$-polytope and $Q$ be a $\mathcal{V}$-polytope 
in $\R^d$, and assume that there exists a box 
$S = \prod_{i=1}^d [-\lambda_i, \lambda_i]$ with $\lambda_i > 0$ and
$P \subseteq S \subseteq \frac{\sqrt{d}}{d}Q$. 
Then the inclusion 
$P \subseteq Q$ is certified in the initial relaxation step.
\end{thm}

To prepare for the proof, we first provide a transitivity property.

\begin{lemma}[Transitivity] \label{lem:transitivity}
~
\begin{enumerate}
  \item 
  Given a $\VV$-polytope $Q$ and $\HH$-polytopes $P$ and $P'$ such that 
  $P'\subseteq P\subseteq Q$. If for a certain $t\geq 2$ the $t$-th step of
  hierarchy~\eqref{eq:putinar} certifies containment of $P$ in $Q$,
  then it also certifies containment of $P'$ in $Q$.
  \item
  Given $\VV$-polytopes $Q$ and $Q'$, and an $\HH$-polytope $P$ such that 
  $P\subseteq Q\subseteq Q'$. If for a certain $t\geq 2$ the $t$-th step of
  hierarchy~\eqref{eq:putinar} certifies containment of $P$ in $Q$,
  then it also certifies containment of $P$ in $Q'$.
\end{enumerate}
\end{lemma}

\begin{proof}
Assume first that $P' = P \cap \{x \in \R^d \, | \, f(x) \ge 0\}$
for an affine function $f:\R^d \to \R$.
Given an sos representation of $\mu(t)-x^T z$ w.r.t. $P$, by setting the
additional sos polynomial $\sigma_{k+1}$ to the zero-polynomial, i.e.
$\sigma_{k+1}\equiv 0$, this yields an sos representation w.r.t. $P'$.

In the general case, starting with $P$, incorporate the defining inequalities
of $P'$ into the representation of $P$ step-by-step. 
In every step the lower bound of the optimal value in~\eqref{eq:putinar} can
not increase. At the end of this process the defining inequalities of $P$ are
all redundant (since $P'\subseteq P$) and thus can be dropped, by
Lemma~\ref{lem:redundant}.
This proves part (1) of the statement. The proof of (2) is analog.
\end{proof}

\begin{proof}[Proof of Theorem~\ref{th:smallinbig}]
Consider $S$ as an $\mathcal{H}$-polytope and let $S'$ be the
box $S' = \frac{d}{\sqrt{d}} S =  \frac{d}{\sqrt{d}}[-\lambda, \lambda]^n$ 
in $\mathcal{V}$-representation. By example~\ref{ex:cube}, the inclusion
$S \subseteq S'$ is certified in the initial relaxation step.
Since $P \subseteq S$ and $S' \subseteq Q$ the transitivity statement
in Lemma~\ref{lem:transitivity} implies that the inclusion
$P \subseteq Q$ is certified
in the initial relaxation step.
\end{proof}

From an optimization viewpoint, such as considering smallest enclosing balls
of a polytope with regard to a polyhedral norm, it is natural to consider
scaled containment problems (cf.\ \cite{gritzmann-klee-containment-survey}).
Theorem~\ref{th:smallinbig} implies the following version of a 
scaled containment.

\begin{cor}[Scaled containment]
Let $P$ be an $\mathcal{H}$-polytope and $Q$ be a $\mathcal{V}$-polytope in
$\R^d$, both containing 0 in the interior. Then there exists a $\lambda > 0$
such that the containment $\lambda P \subseteq Q$ is certified in the initial
relaxation step.
\end{cor}

We conclude with the numerical behavior of relaxation~\eqref{eq:putinar} for
two non-symmetric examples.

\begin{ex} \label{ex:nonsym}
Consider the $\HH$-polytope $P=\{x\in\R^2\ |\ \mathds{1}_4 -Ax \geq 0 \}$ and
the $\VV$-polytopes $Q_1 =\conv B_1$ and $ Q_2 =\conv B_2 $ defined by
\[
  A = \begin{bmatrix} -1 & -1 \\ 0 & -1 \\ 1 & 0 \\ -1 & 1 \end{bmatrix}
  \, ,\quad 
  B_1 =\begin{bmatrix} -1 & 0 & 2 & 2 & -1 \\ 1 & 3 & 1 & -1 & -1\end{bmatrix}
  \, ,\quad 
  B_2 = \begin{bmatrix} -1 & -2 & 1 & 2 & 1 \\ 2 & 0 & -2 & 1 & 2\end{bmatrix}
  \, .
\]
$P$ is contained in both $Q_1$ and $Q_2$ but not strongly contained. 
$Q_1$ and $P$ share infinitely many boundary points, in fact, the boundary of
$Q_1$ contains a facet of $P$. 
$Q_2$ and $P$ intersect in a single vertex. See Figure~\ref{fig:hv-ex}.
Thus in both examples we are not in the situation of
Theorem~\ref{thm:convergence2}. 

\begin{figure}
\centering
\begin{subfigure}{.35\textwidth}
  \centering
  \includegraphics[width=\linewidth]{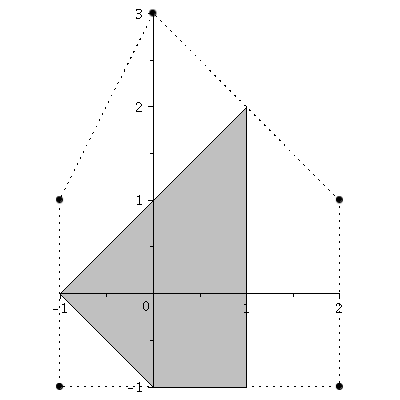}
  \caption{$P$ (grey area) in $Q_1$ (dotted).}
\end{subfigure}\hspace*{.1\textwidth}
\begin{subfigure}{.35\textwidth}
  \centering
  \includegraphics[width=\linewidth]{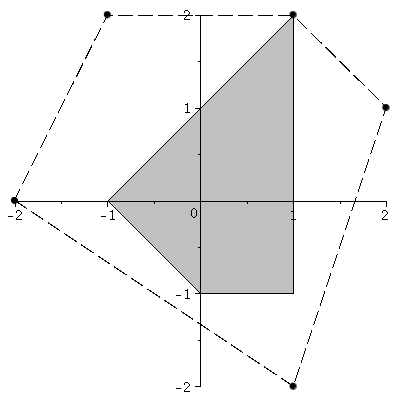}
  \caption{$P$ (grey area) in $Q_2$ (dashed).}
\end{subfigure}
\caption{Two non-symmetric examples as defined in Example~\ref{ex:nonsym}.}
\label{fig:hv-ex}
\end{figure}

The first problem $P\subseteq Q_1$ is not certified in the initial relaxation
step $t=2$. If we scale $P$, the maximum scaling factor for which containment
is certified is $r=0.9271$.
For the second problem $P\subseteq Q_2$, numerically the initial relaxation
does only for a scaling up to $r=0.9996$. 
\end{ex}

\section{Open questions}\label{se:openquestions}

In this paper, we studied algebraic certificates for {\sc Polytope
containment} coming from a sum of squares approach. We close with a short
discussion of open questions. We believe that these questions will be very
relevant in improving the understanding of sum of squares methods for
low-degree geometric problems, such as the one studied here.

For the {\sc Polytope containment} problem, can the structure of the
certificates be better characterized? Such as, what are suitable degree bounds
with regard to {\sc Polytope containment} or, somewhat more general, with
regard to general bilinear programming problems? 

The finite convergence result~\ref{thm:convergence2} is an essential
prerequisite for potential combinatorial accesses to sum of squares
certificates of low-degree problems.
How are $\HH$-to-$\VV$ conversion algorithms (such as the
Fourier-Motzkin-elimination) related to the algebraic certificates of {\sc
Polytope containment}?
Since Theorem~\ref{th:smallinbig} provides a very efficient certification of
{\sc Polytope Containment} in large polytopes, the question also arises in how
far the sum of squares techniques can be effectively combined with existing
combinatorial techniques such as Fourier-Motzkin.

\subsection*{Acknowledgments}

A preliminary version of this paper appeared as a regular contributed talk for
the conference MEGA 2015 (thanks to one of the conference reviewers for
providing an improved sos representation in Example~\ref{ex:cube}).
We would like to thank all the anonymous referees for careful reading and
detailed comments.

\bibliography{hinvarxiv4}
\bibliographystyle{plain}

\end{document}